\documentclass[graybox]{svmult}

\usepackage{type1cm}        
\usepackage{makeidx}         
\usepackage{graphicx}        
\usepackage{multicol}        
\usepackage[bottom]{footmisc}

\usepackage{newtxtext}       %
\usepackage{newtxmath}       

\makeindex             

\def\Z{\mathbb{Z}}

\def\R{\mathbb{R}}

\def\C{\mathbb{C}}
\def\I{{\sf I}}
\def\var{{\rm var}}
\def\cov{{\rm cov}}
\def\E{{\sf E}}

\def\calC{{\cal C}}
\def\L2{{\cal L}_2}

\def\mod{\mbox {\rm mod}}

\def\Re{\mbox {\rm Re} \hspace*{.8truemm}}
\def\Im{\mbox {\rm Im} \hspace*{.8truemm}}

\def\O{{\rm O}}
\def\e{{\rm e}}
\def\i{{\rm i}}
\def\d{{\rm d}}

\def\u{\mathbf}

\def\gvm2{{\rm GvM}_2}
\def\gvmk{{\rm GvM}_k}

\begin{document}

\title*{
Information theoretic results for stationary time series
and the Gaussian-generalized von Mises time series}
\titlerunning{Information theoretic results for stationary time series and Gaussian-GvM time series}
\author{Riccardo Gatto \\ 
January 21, 2021}
\authorrunning{Gatto}
\institute{Riccardo Gatto \at
University of Bern,
Department of Mathematics an Statistics,
Institute of Mathematical Statistics and Actuarial Science,
Alpeneggstrasse 22, 3012 Bern, Switzerland
\\ Email: gatto@stat.unibe.ch 
\\
\\
2010 Mathematics Subject Classification \\
62M20 Inference from stochastic processes: prediction; filtering \\
62H11 Multivariate analysis: directional data; spatial statistics
}

%
\maketitle

\abstract*{This chapter presents some 
novel information theoretic results for the
analysis of stationary time series in the frequency domain.  
In particular, the spectral distribution that 
corresponds to the most uncertain or unpredictable time series 
with some values of the autocovariance function fixed, 
is the {\it generalized von Mises spectral distribution}. 
It is thus a maximum entropy spectral distribution and the 
corresponding stationary time series is called the {\it generalized 
von Mises time series}. The generalized von Mises 
distribution is used in directional statistics
for modelling planar directions that follow a multimodal distribution.
Furthermore, the {\it Gaussian-generalized von Mises times series} is presented
as the stationary time series that maximizes entropies in frequency and time domains,
respectively referred to as {\it spectral} and {\it temporal} {\it entropies}.
Parameter estimation and some 
computational aspects with this time series 
are briefly analyzed.}

\abstract{This chapter presents some
novel information theoretic results for the
analysis of stationary time series in the frequency domain.
In particular, the spectral distribution that
corresponds to the most uncertain or unpredictable time series
with some values of the autocovariance function fixed,
is the {\it generalized von Mises spectral distribution}.
It is thus a maximum entropy spectral distribution and the
corresponding stationary time series is called the {\it generalized
von Mises time series}. The generalized von Mises
distribution is used in directional statistics
for modelling planar directions that follow a multimodal distribution.
Furthermore, the {\it Gaussian-generalized von Mises times series} is presented
as the stationary time series that maximizes entropies in frequency and time domains,
respectively referred to as {\it spectral} and {\it temporal} {\it entropies}.
Parameter estimation and some
computational aspects with this time series
are briefly analyzed.}

\section{Introduction}

Nonstationary data typically have mean, variance, and covariances that 
change significantly over the time.
It is consequently difficult to make reliable predictions or forecasts directly from these
data. For this reason, nonstationary data are transformed to
stationary data, viz. data that possess constant mean, constant variance 
and constant covariance between
any two observations that are separated by any fixed time lag. 
Stationary data are often analyzed
in the frequency domain, where the spectral distribution plays the central role:
it characterizes the correlations between the values of the time series and it
allows for linear predictions. The analysis in the frequency domain is 
particularly interesting for the identification of periodicities of the data.
The first developments of the theory of stationary processes appeared at the end of the
19-th century with the analysis of data in the frequency domain, which is called the spectral analysis.
The alternative analysis in the time domain, viz. based on the covariance function, appeared only later.
The first statistical theory for periodic phenomena was developed by Fisher (1929).
Other early leading contributions to the theory of stationary processes are Cram\'er (1942),
Rice (1944 and 1945) as well as the volumes Cram\'er and Leadbetter (1967) and Yaglom (1962).
An up to date volume on stationary processes is Lindgren (2012) 
and an historical review can be found in Brillinger (1993).
This chapter provides various information theoretic results for spectral distributions 
of stationary processes with discrete time, i.e. stationary time series. 
It recasts the generalized von Mises (GvM) distribution, 
which was introduced in directional statistics as a model for planar directions,
in the context of the spectral analysis of time series.  
It shows that the spectral distribution that corresponds to the most 
uncertain or unpredictable time series and whose autocovariance function agrees with 
some few first predetermined values, for example estimated from a sample, 
is the GvM spectral distribution. It is 
thus a maximum entropy spectral distribution and the corresponding stationary 
time series can be called the generalized von Mises time series. 
The Gaussian stationary time series with GvM spectral distribution is called
Gaussian-GvM. This time series follows 
the maximal entropy principle w.r.t. time and frequency.  
Although some estimation and other computational aspects are briefly analyzed,
this chapter is only a first study of the GvM time series.  

Let $\{X_j\}_{j\in\Z}$ be a complex-valued time series whose
elements belong to a common Hilbert space $\L2$ of square integrable random variables, 
thus $\E\left[|X_j|^2\right]<\infty$, $\forall j\in\mathbb{Z}$.
Its autocovariance function (a.c.v.f.) is given by
\begin{align*}
\psi(j + r,j) & = \cov \left( X_{j+r},X_j \right) = \E \left[ X_{j+r} \overline{X_j}\right] 
- \E[ X_{j+r}]  \E \left[ \overline{X_j} \right], \;
\forall j,r\in\mathbb{Z}. 
\end{align*}
We assume that the time series is weakly stationary, 
which will be shortened to stationary, precisely that   
$\E[X_j]$ and $\psi ( j+ r, j )$ do not depend on $j$, $\forall j,r \in\mathbb{Z}$.
In this case we
denote $\mu = \E[X_j]$, 
\begin{align*}
\psi(r) & = \psi(r,0) = \psi ( j + r, j ), \; \forall j,r \in\mathbb{Z}, 
\end{align*}
and $\sigma^2 = \psi(0)$, for some $\sigma \in (0,\infty)$. A stronger type of
stationarity is the strict stationarity, which requires that 
the double finite dimensional distributions (f.d.d.) of the time series 
are invariant after a fixed time shift, i.e.
$\forall j_1<\ldots < j_n \in\mathbb{Z}$, $r\in\mathbb{Z}$ and $n \ge 1$,      
\begin{align}								\label{e42}
\left(U_{j_1},\ldots,U_{j_n}, V_{j_1},\ldots,V_{j_n}\right) \sim \left(
U_{j_1+r},\ldots,U_{j_n+r},V_{j_1+r},\ldots,V_{j_n+r}\right),
\end{align} 
where $U_j = \Re X_j$ and $V_j = \Im X_j$, $\forall j \in \Z$.
As usual, $E_1  
\sim E_2$ means that the random elements $E_1$ and $E_2$ follow the same distribution.
Stationary time series can be analyzed in the frequency domain, 
precisely through the spectral distribution.
A spectral distribution function (d.f.) is any nondecreasing 
function $F_\sigma: [-\pi , \pi] \to [0,\infty)$ that satisfies 
$F_\sigma(-\pi)=0$ and $F_\sigma(\pi) = \sigma^2$. Thus, the total mass 
of $F_\sigma$ is $\sigma^2$. This d.f. relates to the a.c.v.f. through the equation
\begin{align*}
\psi(r) & = \int_{(-\pi, \pi]} \e^{\i r v} \d F_\sigma(v), \; \forall r \in \mathbb{Z}.
\end{align*}
The simplest nontrivial stationary time series
$\{X_j\}_{j \in \Z}$ is called white noise if it has mean zero and a.c.v.f. 
$$
\psi(r) =
\begin{cases}
   \sigma^2, & \; \mathrm{if} \,\, r=0, \\
          0, & \; \mathrm{if} \,\, r=\pm 1, \pm 2, \ldots,
\end{cases}
$$
for some $\sigma >0$.
All frequencies of $\{X_j\}_{j \in \Z}$ are equally represented, because 
its spectral density is the uniform one with total mass $\sigma^2$, namely
$f_\sigma( \theta ) = \sigma^2/ (2 \pi)$, $\forall \theta \in ( - \pi , \pi ]$.
The term white noise originates from the fact that white color 
reflects all visible wave frequencies of light.
Real-valued time series are used in many applied sciences; refer e.g. to
Brockwell and Davis (1991) or Chatfield (2013).  
However, complex-valued 
time series are often preferred representations of bivariate signals, mainly 
because their compact formulation. They have been applied in various technical domains, 
such as magnetic resonance imaging (cf. e.g. Rowe, 2005) or oceanography 
(cf. e.g. Gonella, 1972).

Spectral distributions of complex-valued time series can be viewed as rescaled circular distributions.
For real-valued time series, the spectral distribution is a rescaled axially symmetric circular distribution.
We recall that a circular distribution is a probability distribution over the circle 
that is used for modelling planar directions as well as periodic phenomena. Two major references
are Mardia and Jupp (2000) and Jammalamadaka and SenGupta (2001).    
For a short introduction cf. e.g. Gatto and Jammalamadaka (2015).
During the last few decades, there has been a considerable amount of theoretical and 
applied research on circular distributions. 

Let $k \in \{1,2,\ldots\}$. A class of circular distributions that possess 
important theoretical properties has densities given by
\begin{eqnarray}                                                                        \label{e94}
\lefteqn{
f_1^{(k)}\left(\theta \mid \mu_1,\ldots,\mu_k,\kappa_1,\ldots,\kappa_k\right)=} \nonumber \\ & &
\frac{1}{2 \pi G_0^{(k)}\left(\delta_1,\ldots,\delta_{k-1},\kappa_1,\ldots,\kappa_k\right)}
\exp\left\{ \sum_{j=1}^{k} \kappa_j \cos j (\theta - \mu_j) \right\},
\end{eqnarray}
$\forall \theta \in (-\pi , \pi]$ (or any other interval of length $2 \pi$),
where $\mu_j \in (-\pi /j, \pi /j]$,
$\kappa_j \ge 0$, for $j=1,\ldots,k$,
\begin{eqnarray*} 
\lefteqn{
G_0^{(k)}(\delta_1,\ldots,\delta_{k-1},\kappa_1,\ldots,\kappa_k) =} \nonumber \\ & &
\frac{1}{2 \pi} \int_{0}^{2 \pi} \exp \{\kappa_1 \cos \theta
+\kappa_2 \cos 2 (\theta + \delta_1) + \ldots +
\kappa_k \cos k (\theta + \delta_{k-1})\} \d \theta,
\end{eqnarray*}
and where
$\delta_{j}=(\mu_{1}-\mu_{j+1}) \mod( 2 \pi/(j+1) )$, for $j = 1,\ldots,k-1$,
whenever $k \ge 2$. 
The circular density (\ref{e94}) for $k\ge 2$ was thoroughly analyzed by
Gatto and Jammalamadaka (2007) and Gatto (2009), 
who called it ``generalized von Mises density of order $k$''
($\gvmk$). Let us denote a circular random variable $\theta$ with that density as
$\theta \sim \gvmk(\mu_1,$ $\ldots,$ $\mu_k,\kappa_1,$ $\ldots,$ $\kappa_k)$.
The ${\rm GvM}_1$ density is the well-known circular normal
or von Mises (vM) density, which represents within 
circular statistics what the normal 
distribution represents in linear statistics. It is given by
$
f_1^{(1)}(\theta \mid \mu_1,\kappa_1) =
\{2 \pi I_0(\kappa_1)\}^{-1}
\exp \{\kappa_1 \cos(\theta -\mu_1)\},
$
$\forall \theta \in (-\pi, \pi]$, where $\mu_1 \in (-\pi,\pi]$, $\kappa_1 \ge0$ and
where $I_{n}(z) = (2 \pi)^{-1} \int_{0}^{2 \pi} \cos n \theta $ $ \exp\{
z \cos \theta\} \d \theta $, $\forall z \in \C$, is the modified Bessel
function of the first kind and integer order $n$ (see e.g. Abramowitz and Stegun, 1972, p. 376).
Compared to the vM, which is axially symmetric and unimodal whenever $\kappa_1 > 0$,
the ${\rm GvM}_2$ distribution
allows for substantially higher adjustability, 
in particular in terms of asymmetry and bimodality.
This makes it a practical circular distributions that has found various applications.
Some recent ones are: 
Zhang et al. (2018), in meteorology, Lin and Dong (2019), in oceanography, Astfalck et al. (2018), 
in offshore engineering, and in Christmas (2014), in signal processing.
The $\gvmk$ spectral density is 
given by $f_\sigma^{(k)} = \sigma^2 f_1^{(k)}$, for some $\sigma \in (0,\infty)$: it
is the $\gvmk$ circular density $f_1^{(k)}$ given in (\ref{e94}) that is rescaled 
to have any desired total mass $\sigma^2$. 
When the $\gvmk$ spectral density is axially symmetric around the null axis, then
the corresponding time series $\{ X_j \}_{j \in \Z}$ is real-valued. 
As shown in Salvador and Gatto (2020a), the $\gvm2$ density
with $\kappa_1,\kappa_2 >0$ 
is axially symmetric iff $\delta_1 = 0$ or $\delta_1 = \pi / 2$. 
In both cases, the axis of symmetry has
angle $\mu_1$ with respect to (w.r.t.) the null direction. 
The $\gvm2$ spectral density has a practical role time series because of its uni- and 
bimodal shape. A complete analysis of the modality of the $\gvm2$ distribution
is given in Salvador and Gatto (2020b). 
Note that in some situations a three-parameter version of the 
$\gvm2$ distribution introduced by
Kim and SenGupta (2013) appears sufficient to model both asymmetric and bimodal data.
The densities of this subclass are obtained by setting $\delta_1 = \pi/4$
and $k=2$ in the $\gvmk$ density (\ref{e94}). However, this subclass does 
not possess the optimality properties of the $\gvm2$ distribution 
that are presented in Section \ref{s2}.  
It is worth mentioning that the GvM spectral distribution has many similarities with 
the exponential model of Bloomfield (1973), which is a truncated Fourier series of the
logarithm of some spectral distribution. Bloomfield motivates the low truncation of the Fourier series  
by the fact that ``the logarithm of an estimated spectral density function is often found to be a fairly
well-behaved function''. A closely related reference is Healy and Tukey (1963).
We should however note that Bloomfield's model if given for real-valued time series only.  

The estimation of the spectral distribution is one of the important problems in
the analysis of stationary time series and of stationary stochastic processes
with continuous time. 
Information theoretic quantities like Kullback-Leibler's information 
(cf. Kullback and Leibler, 1951) or Shannon's entropy (cf. Shannon, 1948) are very
useful in this context. These quantities are usually defined for probability
distributions but they can be considered for distributions with finite mass.
These are spectral distributions and we assume them absolutely continuous (w.r.t. the Lebesgue measure). 
Thus, let $f_\sigma$ and $g_\sigma$ be two spectral densities whose integrals over 
$(-\pi,\pi]$ are both equal to $\sigma^2$. 
The spectral Kullback-Leibler information of $f_\sigma$ w.r.t. $g_\sigma$ is given by
\begin{equation}                                                                        \label{e80}
I(f_\sigma|g_\sigma) = \int_{-\pi}^{\pi} \log \frac{f_\sigma(\theta)}{g_\sigma(\theta)} f_\sigma(\theta) \d \theta
		     = \sigma^2 I(f_1|g_1),
\end{equation}
where $0 \log 0 = 0$ is assumed and where
the support of $f_\sigma$ is included in the support of $g_\sigma$,
otherwise $I(f_\sigma|g_\sigma) = \infty$.
It follows from Gibbs inequality that $I(f_\sigma|g_\sigma)$ is nonnegative, precisely 
$$I(f_\sigma|g_\sigma) \ge 0,$$
for all possible spectral densities $f_\sigma$ and $g_\sigma$, with
equality iff $f_1=g_1$ a.e. The Kullback-Leibler information is also called
relative entropy, Kullback-Leibler divergence or distance, eventhough
it is not a metric. Thus (\ref{e80}) is a measure of divergence for 
distributions with same total mass $\sigma^2$.
Shannon's entropy can be defined for the spectral density $f_\sigma$ by
\begin{equation}                                                                        \label{e92}
S(f_\sigma) = - \int_{-\pi}^{\pi} \log \frac{f_\sigma(\theta)}{(2 \pi)^{-1} \sigma^2} f_\sigma (\theta) \d \theta
	    = - I ( f_\sigma | u_\sigma)  = - \sigma^2 I (f_1| u_1 ),
\end{equation}
where $u_\sigma$ is the uniform density with total mass $\sigma^2$ over $(-\pi,\pi]$, viz.
$u_\sigma = \sigma^2 /(2 \pi) \I_{(-\pi,\pi]}$, $\I_A$ denoting the indicator of set $A$. 
Shannon's entropy of the circular density $f_1$ over $(-\pi,\pi]$ is originally defined as 
$- \int_{-\pi}^{\pi} \log f_1(\theta) f_1 (\theta) \d \theta
= -(2 \pi)^{-1} - I ( f_1 | u_1 )$. It
measures the uncertainty inherent in the probability distribution with density $f_1$.
Equivalently, $S(f_1)$ measures the expected amount of information
gained on obtaining an observation from $f_1$, based on the principle
that the rarer an event, the more informative its occurrence. The spectral entropy 
defined in (\ref{e92})
differs slightly differs the original formula of
Shannon's entropy for probability distributions:
inside the logarithm, $f_\sigma$ is divided by the uniform density with total mass $\sigma^2$.
With this modification the spectral entropy becomes scale invariant w.r.t. $\sigma^2$,
just like the spectral Kullback-Leibler information (\ref{e80}).
The spectral entropy satisfies
$S(f_\sigma) \le 0$, with equality iff $f_\sigma = u_\sigma$ a.e.
This follows from Gibbs inequality.

The topics of the next sections of this chapter are the following. 
Section \ref{s2} provides information theoretic
results for spectral distributions and introduces the related GvM and the Gaussian-GvM time series.
Section \ref{s21} gives general definitions and concepts.
Section \ref{s22} provides the optimal spectral distributions under constraints on the a.c.v.f.
The GvM spectral distribution appears as the one that maximizes Shannon's entropy under
constraints on the first few values of the a.c.v.f. 
Section \ref{s23} motivates the
Gaussian-GvM time series from the fact that it 
follows the maximal entropy principle in both time and frequency domains.
Section \ref{s3} provides some computational aspects. 
Section \ref{s31} gives some series expansions for integral functions appearing in the 
context of the $\gvm2$ time series, which is one the most relevant case. 
The estimation problem of the GvM spectral distribution is presented in Section \ref{s32}.
After presenting some general consideration and some related results on maximum entropy
spectral estimation, an estimator to the parameters of the GvM spectral distribution is provided.
Section \ref{s33} provides an expansion for the GvM spectral d.f.  
Some short concluding remarks are given in Section \ref{s4}. 

\section{The GvM and the Gaussian-GvM time series}					\label{s2}

Section \ref{s21} gives some general facts on time series and defines 
the GvM time series. 
Section \ref{s22} provides information theoretic results for spectral distributions.
An important result is that the GvM spectral distribution maximizes the entropy
under constraints on the a.c.v.f. Section \ref{s23} proposes the
Gaussian-GvM time series based on the fact that it             
follows the maximal entropy principle in both time and frequency domains, under the
same constraints. 

\subsection{General considerations}						\label{s21}  

Two central theorems of spectral analysis of time series are the following.
The first one is Herglotz theorem: 
\begin{itemize}
\item[]  
{\it 
$\psi:\mathbb{Z} \rightarrow \mathbb{C}$ is nonnegative definite (n.n.d.)}\footnote{
The function $\psi : \mathbb{R} \rightarrow \mathbb{C} $ is
n.n.d.
if $\sum_{i=1}^n \sum_{j=1}^n $ $c_i \overline{c}_j f( x_i - x_j ) \geq 0$,
$\forall x_1, \ldots , x_n \in \mathbb{R}$, $c_1, \ldots , c_n \in \mathbb{C}$
and $ n \ge 1$.
\\
Any n.n.d. function $f$ is Hermitian, i.e. $f(-x) = \overline{f(x)}$, $\forall x \in 
\R$.}{\it
$\Leftrightarrow$
$\psi(r) = \int_{(-\pi, \pi]} \e^{\i r v} \d F_\sigma(v)$, $\forall r \in\mathbb{Z}$,
for some 
d.f. $F_\sigma$ over $[-\pi , \pi]$, with $F_\sigma(-\pi)=0$ and $\sigma^2=
F_\sigma(\pi) \in (0,\infty)$.
}
\end{itemize}
The second theorem is a characterization of the a.c.v.f.: 
\begin{itemize}
\item[] 
{\it 
$\psi:\mathbb{Z}\rightarrow \mathbb{R}$ is the a.c.v.f. of a (strictly)
stationary complex-valued time series $\Leftrightarrow$ $\psi$ is n.n.d.
}
\end{itemize}

These two theorems tell that if we consider the spectral d.f. $F^{(k)}_\sigma = \sigma^2 F_1^{(k)}$,
where $F_1^{(k)}$ is 
the $\gvmk$ d.f. with density $f_1^{(k)}$ given by (\ref{e94}), then there exists a stationary time
series $\{ X_j \}_{j \in \Z}$ with spectral d.f. $F_\sigma^{(k)}$ and density $f_\sigma^{(k)} = \sigma f_1^{(k)}$ 
that we call GvM or, more precisely, $\gvmk$ time series.
Thus the $\gvmk$ time series is stationary by definition, it has variance 
$F_\sigma^{(k)}(\pi) = \sigma^2$ and it is generally complex-valued, unless
the $\gvmk$ spectral distribution is axially symmetric around the origin. 

The complex-valued $\gvmk$ stationary time series 
$\{ X_j \}_{j \in \Z}$ can be chosen with mean zero, variance $\sigma^2$
and Gaussian, meaning that the double f.d.d. given in (\ref{e42}) are Gaussian. 
In this case, the distribution of $\{ X_j \}_{j \in \Z}$ is
however not entirely determined by its a.c.v.f. $\psi^{(k)}$ or, 
alternatively, by its spectral d.f. $F_\sigma^{(k)}$.
(The formula for the a.c.v.f. is given later in Corollary \ref{c22}.4.)
In order to entirely determine this distribution, one also needs  
the so-called pseudo-covariance $\E[ X_{j+r} X_j]$, $\forall j,r \in \Z$.
So an arbitrary 
Gaussian, with mean zero and (weakly) stationary time series $\{ X_j \}_{j \in \Z}$
is not necessarily strictly stationary:
$\{ X_j \}_{j \in \Z}$ is strictly stationary iff the covariance 
$\E\left[ X_{j+r} \overline{X_j}\right]$ and the pseudo-covariance $\E[ X_{j+r} X_j]$ 
do not depend on $j \in \Z$, $\forall r \in \Z$.
This is indeed equivalent to the independence on $j \in \Z$ of
\begin{align}								\label{e129} 
  \psi_{UU}(r) & = \E[ U_{j+r} U_j ], \; 
  \psi_{VV}(r) = \E[ V_{j+r} V_j ], \nonumber \\ 
  \psi_{UV}(r) & = \E[ U_{j+r} V_j ] \; \text{ and } 
  \psi_{VU}(r) = \E[ V_{j+r} U_j ],
  \; \forall r \in \Z,
\end{align}
where $U_j = \Re X_j$ and $V_j = \Im X_j$, $\forall j \in \Z$.
Under this independence on $j \in \Z$, we have
$\psi_{VU}(r) = \psi_{UV}(-r)$, $\forall r \in \Z$. 
However, according Herglotz theorem, if the a.c.v.f. $\psi^{(k)}$ 
is obtained by Fourier inversion 
of the $\gvmk$ spectral density, then it is n.n.d. By the above characterization   
of the a.c.v.f., a strictly stationary $\gvmk$ time series always exists. 
The existence of a particular (precisely radially symmetric) 
strictly stationary Gaussian-$\gvmk$ time series 
that satisfies some constraints on the a.c.v.f. is shown in Section \ref{s23}. 

Next, for any given Gaussian-$\gvmk$ time series with spectral d.f. $F_\sigma^{(k)}$, 
there exists a spectral process $\{ Z_\theta \}_{\theta \in [-\pi,\pi]}$ that is complex-valued and 
Gaussian.
We remind that the process of the frequencies $\{ Z_\theta \}_{\theta\in [-\pi,\pi]}$ is 
defined through the mean square stochastic integral
\begin{align}								\label{e145}  
X_j & = \int_{(-\pi,\pi]} \e^{ \i \theta j } \d Z_\theta, \; \text{ a.s.,}
\; \forall j \in \Z,
\end{align}
and by the following conditions:
$\E [ Z_\theta ] = 0$, $\forall \theta \in [-\pi,\pi]$,
$\E \left[ \left(Z_{\theta_2} - Z_{\theta_1} \right) \overline{ \left( Z_{\theta_4} - Z_{\theta_3} \right)}
\right]$ $= 0$, $\forall - \pi \le \theta_1 < \theta_2 < \theta_3 < \theta_4 \le \pi$,
viz. it has orthogonal increments, and
\begin{align}											\label{e184}
\E \left[ \left| Z_{\theta_2} - Z_{\theta_1} \right|^2 \right] & =  
F_\sigma^{(k)}(\theta_2) - F_\sigma^{(k)}(\theta_1), \; \forall - \pi \le \theta_1 < \theta_2 \le \pi.
\end{align}

There are several reasons for considering the Gaussian-GvM time series. A practical one is
that their simulation can be done with the algorithms presented in 
Chapter XI of Asmussen and Glynn (2007). One of these algorithms the decomposition 
(\ref{e145}). A theoretical reason for considering normality is that 
it leads to a second maximal entropy principle,
this one no longer in the frequency domain but in the time domain. 
We pursue this explanation on the temporal entropy in Section \ref{s23}.   

\subsection{Spectral Kullback-Leibler information and entropy}                                         \label{s22}

Let $g_\sigma$ be the spectral density of some stationary time series with 
variance $\sigma^2$, for some $\sigma \in (0,\infty)$. 
For a chosen $k \in \{ 1,2,\ldots\}$, consider the $r$-th a.c.v.f. condition or constraint
\begin{align}								\label{e199} 
\calC_r: & \int_{-\pi}^{\pi} \e^{\i r \theta} g_\sigma(\theta) \d \theta = \psi_r, 
\end{align}
for some $\psi_r \in \C$ satisfying $|\psi_r| \le \sigma^2$,
for $r=1,\ldots,k$, and such that the $(k+1) \times (k+1)$ matrix
\begin{align}									\label{e206}  
\left(
\begin{matrix}
\sigma^2          & \psi_1                & \ldots  & \psi_k     \\
\overline{\psi_1} & \sigma^2              & \ldots  & \psi_{k-1} \\
\vdots            & \vdots                & \ddots & \vdots     \\
\overline{\psi_k} & \overline{\psi_{k-1}} & \ldots  & \sigma^2
\end{matrix}
\right)
\end{align}
is n.n.d., for $k=1,2,\ldots$. 
One can re-express these conditions as 
\begin{align}                                                           \label{e201}
\calC_r: &  
\int_{-\pi}^{\pi}  \cos r \theta \, g_\sigma(\theta) \d \theta  = \nu_r \; \text{ and } \; 
\int_{-\pi}^{\pi} \sin r \theta \, g_\sigma(\theta) \d \theta = \xi_r, 
\end{align} 
where $\nu_r = \Re \psi_r$ and $\xi_r = \Im \psi_r$, giving thus $\nu_r^2 + \xi_r^2 \le \sigma^2$,
for $r=1,2,\ldots$, and with n.n.d. matrix (\ref{e206}), for a chosen $k \in \{ 1, 2, \ldots \}$.
One can encounter the two following practical problems.

In an applied field where a specific spectral density $h_\sigma$ is traditionally used
(refer to comments in Section \ref{s4}), one may search for the
the spectral density $g_\sigma$ that satisfies $\calC_r$, given in
(\ref{e199}), for $r=1,\ldots,k$, and that
is the closest to the traditional density $h_\sigma$. 

Alternatively, the spectral density $g_\sigma$ is unknown but the
values of $\psi_1, \ldots , \psi_k$ are available, either
because they constitute a priori knowledge about the time series
or because they are obtained from a sample of the stationary time series.
In this second case, the values of $\psi_1,\ldots,\psi_k$ can be obtained
by taking them equal to the corresponding values of the
empirical or sample a.c.v.f.
For the sample $X_1,\ldots,X_n$ of the time series, the sample a.c.v.f.
is given by the Hermitian function
\begin{align}									\label{e187} 
\hat{\psi}_n (r)
= \frac{1}{n} \sum_{j=1}^{n-r} (X_{j+r} - M_n ) \overline{(X_j - M_n)}
\;\;
\mathrm{and}
\;\;
\hat{\psi}_n (-r) = \overline{\hat{\psi}_n(r)}, \;\;
\text{ for } \, r = 0,\ldots , n-1,
\end{align}
where $M_n = n^{-1 }\sum_{j=1}^n X_j$.
Thus we set $\psi_r = \hat{\psi}_n(r)$, for $r=1,\ldots,k$ and for $k \le n-1$.
Note that the matrix (\ref{e206}) is n.n.d. in this case. 

Theorem \ref{c100} below addresses the first of these two problems and 
it is the central part of this article. 
The second problem is addressed by Corollary \ref{c22}.
The following definitions are required. 
For $k = 1,2,\ldots$ and for an arbitrary circular density $g_1$, define the following integral 
functions:
$$G_r^{(k)}
\left(\delta_1,\ldots,\delta_{k-1},\kappa_{1},\ldots,\kappa_{k};g_1\right) =
\;\;\;\;\;\;\;\;\;\;\;\;\;\;\;\;\;\;\;\;\;\;\;\;\;\;\;\;\;\;\;\;\;
\;\;\;\;\;\;\;\;\;\;\;\;\;\;\;\;\;\;\;\;\;\;\;\;\;\;\;\;\;\;\;\;\;
\;\;\;\;\;\;\;\;\;\;\;\;
$$
$$
\int_0^{2 \pi} \cos r \theta \,
\exp\left\{\kappa_{1} \cos \theta
+ \kappa_{2} \cos 2 (\theta + \delta_1)
+ \ldots + \kappa_{k} \cos k (\theta + \delta_{k-1})
\right\} g_1 (\theta) 
\d \theta,
$$
$$H_r^{(k)}
(\delta_1,\ldots,\delta_{k-1},\kappa_{1},\ldots,\kappa_{k};g_1) =
\;\;\;\;\;\;\;\;\;\;\;\;\;\;\;\;\;\;\;\;\;\;\;\;\;\;\;\;\;\;\;\;\;
\;\;\;\;\;\;\;\;\;\;\;\;\;\;\;\;\;\;\;\;\;\;\;\;\;\;\;\;\;\;\;\;\;
\;\;\;\;\;\;\;\;\;\;\;\;
$$
$$
\int_0^{2 \pi} \sin r \theta \,
\exp\{\kappa_{1} \cos \theta
+ \kappa_{2} \cos 2 (\theta + \delta_1)
+ \ldots + \kappa_{k} \cos k (\theta + \delta_{k-1})
\} g_1 (\theta)
\d \theta,
$$
\begin{align*}                                                                        \label{e160}
A_r^{(k)}
\left(\delta_1,\ldots,\delta_{k-1},\kappa_{1},\ldots,\kappa_{k};g_1\right) & =
\frac{G_r^{(k)}
\left(\delta_1,\ldots,\delta_{k-1},\kappa_{1},\ldots,\kappa_{k};g_1\right)}
{G_0^{(k)}\left(\delta_1,\ldots,\delta_{k-1},\kappa_{1},\ldots,\kappa_{k};g_1\right)}
\end{align*}
and
\begin{align*} 
B_r^{(k)}
\left(\delta_1,\ldots,\delta_{k-1},\kappa_{1},\ldots,\kappa_{k};g_1\right) & =
\frac{H_r^{(k)}
\left(\delta_1,\ldots,\delta_{k-1},\kappa_{1},\ldots,\kappa_{k};g_1\right)}
{G_0^{(k)}
\left(\delta_1,\ldots,\delta_{k-1},\kappa_{1},\ldots,\kappa_{k};g_1\right)},
\end{align*}
for $r=1,\ldots,k$, where  
$\delta_{j}=(\mu_{1}-\mu_{j+1}) \mod( 2 \pi/(j+1) )$, for $j = 1,\ldots,k-1$
and
$\kappa_{1},\ldots,\kappa_{k}$ $\ge0$.
For these constants we make the conventions that the arguments
$\delta_1,\ldots,\delta_{k-1}$ vanish when $k=1$ and 
that the argument $g_1$ is omitted when equal to the circular uniform density $u_1$.
For example, $G_0^{(1)}(\kappa_1) = (2 \pi)^{-1} \int_0^{2 \pi} \e^{\kappa_1 \cos \theta} 
\d \theta = I_0(\kappa_1)$.
Define the matrix of counter-clockwise rotation of angle $\alpha$ as
\begin{equation}  		                                                      \label{e109}
\u{R}(\alpha) =
   \left( \begin{array}{cc} \cos \alpha  & - \sin \alpha \\
                            \sin \alpha  &   \cos \alpha
   \end{array} \right).
\end{equation}
\begin{theorem}[Kullback-Leibler closest spectral distribution]		  	\label{c100} 
\noindent 
Let $\sigma \in (0,\infty)$ and let $g_\sigma$ and $h_\sigma$ be two spectral densities
with total mass $\sigma^2$.
\begin{enumerate} 
\item The spectral density $g_\sigma$ that satisfies $\calC_r$, given in
(\ref{e199}), for $r=1,\ldots,k$, and that
is the closest to another spectral density $h_\sigma$, in the sense of minimizing
the Kullback-Leibler information $I(g_\sigma|h_\sigma)$,
is the exponential tilt of $h_\sigma$ that takes the form 
\begin{eqnarray}                                                                        \label{e100}
\lefteqn{
g_\sigma(\theta) =} \nonumber \\ & &
\frac{1}{G^{(k)}_0(\delta_{1},\ldots,\delta_{k-1},\kappa_1,\ldots,\kappa_k;h_1)}
\exp\left\{ \sum_{j=1}^k
\kappa_j \cos j (\theta-\mu_j)
\right\} h_\sigma(\theta), 
\end{eqnarray}
$\forall \theta \in (-\pi,\pi]$, where
$\delta_{j}=(\mu_{1}-\mu_{j+1}) \mod( 2 \pi/(j+1))$, for $j = 1,\ldots,k-1$, 
$\mu_j \in (-\pi /j, \pi /j]$ and $\kappa_j \ge 0$, for $j=1,\ldots,k$. 
The values of these parameters are the solutions of
\begin{equation}                       		                             \label{e108}
   \left( \begin{array}{c}\nu_r \\ \xi_r
   \end{array} \right) = \sigma^2 \u{R}(r\mu_1)
   \left( \begin{array}{c} A_r^{(k)}(\delta_1,\ldots,\delta_{k-1},\kappa_1,\ldots,\kappa_{k};h_1) \\
                           B_r^{(k)}(\delta_1,\ldots,\delta_{k-1},\kappa_1,\ldots,\kappa_{k};h_1)
   \end{array} \right), 
\end{equation}
where $\u{R}(r \mu_1)$ denotes the rotation matrix (\ref{e109}) at 
$\alpha = r \mu_1$ and where $\nu_r$ and $\xi_r$ are given by (\ref{e201}), 
for $r=1,\ldots,k$.
\item For any spectral density $g_\sigma$ that satisfies $\calC_r$, for $r=1,\ldots,k$,
the minimal Kullback-Leibler information of $g_\sigma$ w.r.t. $h_\sigma$ 
is given by 
\begin{align}                                                                \label{e105}
- \sigma^2 \log G^{(k)}_0(\delta_1,\ldots,\delta_{k-1},\kappa_1,\ldots,\kappa_k;h_1) +
\sum_{r=1}^{k} \kappa_r (\nu_r & \cos r \mu_r + \xi_r \sin r \mu_r ) \nonumber \\ 
	       			& \le I(g_\sigma|h_\sigma),
\end{align}
with equality iff $g_\sigma$ is a.e. given by (\ref{e100}), where
the values of the parameters $\mu_j \in (-\pi /j, \pi /j]$ and $\kappa_j \ge 0$, for $j=1,\ldots,k$,
are solutions of (\ref{e108}).
\end{enumerate} 
\end{theorem}

Theorem \ref{c100} is a rather direct
consequence or generalization of Theorem 2.1 of Gatto (2007), in
which the trigonometric moments are replaced by the a.c.v.f.
and the circular distribution is replaced by the spectral distribution.
Indeed, along with the generalization of the circular distribution
to the spectral distribution, the a.c.v.f. of a stationary time series
generalizes the trigonometric moment. Precisely, the $r$-th trigonometric 
moment of the circular random variable $\theta$ with density $g_1$ is given by
\begin{align}										\label{e170} 
\varphi_r & = \gamma_r + \i \sigma_r = \E \left[ \e^{\i r \theta} \right] = \int_{-\pi}^\pi \e^{\i r \theta}
g_1(\theta) \d \theta,
\end{align}
for some $\gamma_r, \sigma_r \in \R$ and
$\forall r \in \Z$, whereas the a.c.v.f. of the stationary time series with the spectral density
$g_\sigma = \sigma^2 g_1$ is given by
\begin{align}                                                                           \label{e182}
\psi(r) & = \sigma^2 \varphi_r = \sigma^2 ( \gamma_r + \i \sigma_r ), \; \forall r \in \Z.
\end{align}
Clearly, $\psi (0) = \sigma^2$ and $|\psi (r)| \le \psi (0)$, $\forall r \in \Z$.
The claim that (\ref{e182}) is indeed the a.c.v.f. of a stationary time series
is rigorously justified by the above mentioned Herglotz theorem and  
characterization of the a.c.v.f. 

In the context of the justification of Theorem \ref{c100}.1,
we can note that an equivalent expression for (\ref{e108}) is given by 
\begin{align}										\label{e119} 
\psi_r & = \sigma^2 \e^{\i r \mu_1}
	\left\{A_r^{(k)}(\delta_1,\ldots,\delta_{k-1},\kappa_1,\ldots,\kappa_{k};h_1) + \i
	  B_r^{(k)}(\delta_1,\ldots,\delta_{k-1},\kappa_1,\ldots,\kappa_{k};h_1) \right\},
\end{align}
which can be seen equivalent to $\calC_r$, for $r=1,\ldots,k$.

A major consequence of Theorem \ref{c100} is 
that the $\gvmk$ spectral distribution is a maximum entropy distribution.
This fact and related results are given in Corollary \ref{c22}. 
\begin{corollary}[Maximal Shannon's spectral entropy distribution]		\label{c22}
\noindent 
Let $\sigma \in (0,\infty)$ and $g_\sigma$ a spectral density with total mass $\sigma^2$. 
\begin{enumerate}
\item The spectral density $g_\sigma$ that maximizes Shannon's entropy $S( g_\sigma )$ under
$\calC_r$, given in (\ref{e199}), 
for $r=1,\ldots,k$, is the $\gvmk (\mu_1,\ldots,\mu_{k},\kappa_1,\ldots,\kappa_{k})$
density multiplied by $\sigma^2$, viz. 
$f_\sigma^{(k)}$ $(\cdot | \mu_1,\ldots,\mu_{k},\kappa_1,\ldots,\kappa_{k})$,
where $\mu_j \in (-\pi /j, \pi /j]$ and $\kappa_j \ge 0$, for $j=1,\ldots,k$.
The values of these parameters are determined by (\ref{e108}).
\item If $g_\sigma$ is a spectral density satisfying
$\calC_r$, for $r=1,\ldots,k$, then its entropy is bounded from above as follows,
\begin{equation*}
S(g_\sigma) \le \sigma^2 \log G_0^{(k)}(\delta_1,\ldots,\delta_{k-1},\kappa_1,\ldots,\kappa_k)
- \sum_{r=1}^{k} \kappa_r (\nu_r \cos r \mu_r + \xi_r \sin r \mu_r ),
\end{equation*}
with equality iff $g_\sigma = f_\sigma^{(k)}(\cdot | \mu_1,\ldots,\mu_{k},\kappa_1,\ldots,\kappa_{k})$ a.e. 
The values of the parameters are determined by 
(\ref{e108}) with $h_1 = u_1$, i.e. the circular uniform density, where
$\nu_r$ and $\xi_r$ are given by (\ref{e201}), for $r = 1,\ldots,k$.
\item The entropy of the $\gvmk (\mu_1,\ldots,\mu_{k},\kappa_1,\ldots,\kappa_{k})$ spectral density
with total mass $\sigma^2$ is given by the formula
\begin{align*}
S \left( f_\sigma^{(k)} \right)  = & \sigma^2 \Bigg\{ 
\log G_0^{(k)}(\delta_1,\ldots,\delta_{k-1},\kappa_1,\ldots,\kappa_k) 
	- \kappa_1 A_1^{(k)}(\delta_1,\ldots,\delta_{k-1},\kappa_1,\ldots,\kappa_k)   \\ 
	&  - \sum_{r=2}^{k} \kappa_r \big[ A_r^{(k)}(\delta_1,\ldots,\delta_{k-1},\kappa_1,\ldots,\kappa_k) \cos r \delta_{r-1} \\
	&  - B_r^{(k)}(\delta_1,\ldots,\delta_{k-1},\kappa_1,\ldots,\kappa_k) \sin r \delta_{r-1} \big] \Bigg\},
\end{align*}   
where $\sum_{r=2}^{k}$ is defined as null whenever $k<2$.
\item The a.c.v.f. $\psi^{(k)}$ 
of the $\gvmk (\mu_1,\ldots,\mu_{k},\kappa_1,\ldots,\kappa_{k})$ spectral distribution can be
obtained by
$$
   \left( \begin{array}{c} \Re \psi^{(k)}(r) \\ \Im \psi^{(k)} (r)
          \end{array} \right)
   = \sigma^2 \u{R}(r\mu_1)
   \left( \begin{array}{c} A_r^{(k)}(\delta_1,\ldots,\delta_{k-1},\kappa_1,\ldots,\kappa_{k}) \\
                           B_r^{(k)}(\delta_1,\ldots,\delta_{k-1},\kappa_1,\ldots,\kappa_{k})
          \end{array} \right)
$$
and $\psi^{(k)} (-r) = \overline{\psi^{(k)}(r)}$, for $r=1,2,\ldots$.
\end{enumerate} 
\end{corollary}

Corollary \ref{c22} can be obtained from Theorem \ref{c100} as follows.
Theorem \ref{c100}.1 and the relation between Kullback-Leibler information and entropy
(\ref{e92}) tell that the $\gvmk$ spectral distribution maximizes the entropy, 
under the given constraints on the a.c.v.f. 
The upper bound for the entropy
of a circular distribution satisfying the given constraints is provided by
Theorem \ref{c100}.2. Thus, by considering $h_1 = u_1$ , we obtain
the parts 1 and 2.
The part 3 is a consequence of the part 2. 
It is obtained by replacing $\nu_r$ and $\xi_r$, for $r=1,\ldots,k$, that appear 
in the upper bound of the entropy,
by expressions depending on the parameters of the $\gvmk$ distribution, 
through the identity (\ref{e108}). 

Thus, when partial prior information in the form of $\calC_r$, 
for $r=1,\ldots,k$, is
available and it is desired to determine the most noninformative spectral 
distribution that satisfies the known prior
information, then the $\gvmk$ spectral distribution is the optimal one.
It is in fact the most credible distribution, or the one that nature would have generated, 
when the prior information and only that information would be available. 
Maximal entropy distributions are important in many contexts. 
In statistical mechanics, the choice of a maximum entropy
distribution subject to constraints is a classical approach referred to as 
the maximum entropy principle. 
One can find various studies on spectral distributions with maximal entropy.
It is explained in Section \ref{s32} that the autoregressive model of 
order $k$ (AR($k$)) maximizes an alternative 
entropy among all time series satisfying ${\cal C}_r$, for $r=1,\ldots,k$.
Franke (1985) showed that the autoregressive and moving average time series (ARMA)
maximizes that entropy among all time series satisfying these same constraints 
on the a.c.v.f. and additional constraints on the impulse responses.
Further properties on these optimal ARMA time series can be found in Huang (1990). 
There are many other references on spectral distributions with maximal entropy:
Burg (1978), Kay and Marple (1981), Laeri (1990), etc.

The simplest situation is the following.
\begin{example}[{\rm vM} spectrum]
Corollary \ref{c22}.3 with $k=1$ yields the entropy of the vM spectral distribution,
\begin{align*}
S \left( f_\sigma^{(1)} \right) & = \sigma^2 \left\{ \log G_0^{(1)}(\kappa_1) - \kappa_1 A_1^{(1)}(\kappa_1) 
\right\} = \sigma^2 \left\{ \log I_0(\kappa_1) - \kappa_1 \frac{I_1(\kappa_1)}{I_0(\kappa_1)} \right\},
\end{align*}
for $\kappa_1 \ge 0$. 
By noting that $B_r^{(1)}(\kappa_1)=0$, for $r=1,2,\ldots$,
Corollary \ref{c22}.4 with $k=1$ gives the a.c.v.f. of the vM spectral distribution as 
\begin{align}									\label{e159}
   \left( \begin{array}{c} \Re \psi^{(1)}(r) \\ \Im \psi^{(1)} (r)
          \end{array} \right)
   = \sigma^2 A_r^{(1)}(\kappa_1) 
          \left( \begin{array}{c} 
			\cos r \mu_1 \\ \sin r \mu_1
          \end{array} \right)
   = \sigma^2 \frac{ I_r(\kappa_1) }{ I_0(\kappa_1) }
   \left( \begin{array}{c} 
                        \cos r \mu_1 \\ \sin r \mu_1
          \end{array} \right),
\end{align}
and $\psi^{(1)} (-r) = \overline{\psi^{(1)}(r)}$, for $r=1,2,\ldots$.
When $\kappa_1 > 0$,
the vM spectral distribution is axially symmetric about the origin iff $\mu_1 = 0$.
In other terms and according to (\ref{e159}), 
the ${\rm GvM}_1$ or vM time series is real-valued iff $\mu_1 = 0$.  
\end{example}

\subsection{Temporal entropy}  			                                       \label{s23}

This section provides a strictly stationary Gaussian-GvM 
time series that follows the maximal entropy principle in the 
time domain, in addition to the maximal entropy principle in the frequency domain, 
under the previous constraints on the a.c.v.f.
Consider the complex-valued Gaussian time series $\{X_j\}_{j \in \Z}$ 
in $\L2$ that is strictly stationary with mean zero. 
This time series is introduced at the end of Section \ref{s21}.
Define $U_j = \Re X_j$ and $V_j = \Im X_j$, $\forall j \in \Z$.
Let $n \ge 1$ and $j_1<\ldots < j_n \in\mathbb{Z}$.
Consider the random vector $(U_{j_1},\ldots,U_{j_n},V_{j_1},\ldots,V_{j_n})$
and denote by $p_{j_1,\ldots,j_n}$ its joint density. 
Thus $p_{j_1,\ldots,j_n}$ is the $2n$-dimensional normal density 
with mean zero and $2 n \times 2 n$ covariance matrix 
\begin{align}									\label{e97}  
\u{\Sigma}_{j_1,\ldots,j_{n}} & = \var \left( \left( U_{j_1},\ldots,U_{j_n},V_{j_1},\ldots,V_{j_n} \right) \right) 
= \E \left[ \left(
\begin{array}{cc}
\u{U} {\u{U}}^\top &  {\u{U}} {\u{V}}^\top \\
\u{V} {\u{U}}^\top &  \u{V} {\u{V}}^\top
\end{array}
\right) \right], 
\end{align} 
where $\u{U} = \left( U_{j_1},\ldots,U_{j_n} \right)^\top$ and $\u{V} = \left( V_{j_1},\ldots,V_{j_n} \right)^\top$.
According to (\ref{e129}), the elements of $\u{\Sigma}_{j_1,\ldots,j_{n}}$ are given by 
\begin{align*}      
\E [ U_{j_l} U_{j_m} ] & = \psi_{UU} ( j_l - j_m ), \;
\E [ V_{j_l} V_{j_m} ]   = \psi_{VV} ( j_l - j_m ), \nonumber \\
\E [ U_{j_l} V_{j_m} ] & = \psi_{UV} ( j_l - j_m ) \; \text{ and } \;
\E [ V_{j_l} U_{j_m} ]   = \psi_{VU} ( j_l - j_m ),
\end{align*}
with $\psi_{VU} ( j_l - j_m ) = \psi_{UV} ( j_m - j_l )$,
for $l,m = 1,\ldots,n$.
Because $\u{\Sigma}_{j_1,\ldots,j_{n}}$ depends on $j_1,\ldots,j_{n}$ only through 
$l_1 = j_2 - j_1, \ldots, l_{n-1} = j_n - j_{n-1}$, we consider the alternative notation
$\u{\Sigma}^{l_1,\ldots,l_{n-1}} = \u{\Sigma}_{j_1,\ldots,j_{n}}$. 

An important subclass of complex-valued normal random vectors is made by the 
radially symmetric ones, which is
obtained by setting the mean and the pseudo-covariance matrix equal to zero.
That is, the Gaussian vector 
$\u{X} = (X_{j_1},\ldots,X_{j_n})^\top$, where
$X_{j_l} =  U_{j_l} + \i V_{j_l}$, for $l = 1, \ldots, n$,
is radially symmetric iff $\E[ \u{X} ] = \u{0}$ and
$\E \left[ \u{X} \u{X}^\top \right] = \u{0}$. 
A radially symmetric complex normal random vector 
$\u{X}$ is characterized by the fact that,
$\forall \theta \in (-\pi,\pi]$, 
$\e^{ \i \theta} \u{X} \sim \u{X}$. Because these vectors and the related processes  
are often used in signal processing, we consider them in this section.  
 
More generally, by assuming neither stationarity nor normality, we  
define the temporal entropy of the complex-valued time series 
$\{X_j\}_{j \in \Z}$ at times $j_1 < \ldots < j_n \in \Z$ 
in terms of Shannon's entropy of  
$\left( U_{j_1},\ldots,U_{j_n},\right.$ $\left. V_{j_1},\ldots,V_{j_n} \right)$, precisely as 
\begin{align}									\label{e93} 
T_{j_1,\ldots,j_n} = - \int_{-\infty}^\infty & \ldots \int_{-\infty}^\infty    
		       \log p_{j_1,\ldots,j_n}(u_1,\ldots,u_{n},v_1,\ldots,v_n) \nonumber \\
		     & p_{j_1,\ldots,j_n}(u_1,\ldots,u_{n},v_1,\ldots,v_n) \d u_1 \ldots \d u_n 
		       \d v_1 \ldots \d v_n ,
\end{align}
whenever the density $p_{j_1,\ldots,j_n}$ exists. 
Under strict stationarity, the temporal entropy (\ref{e93}) becomes invariant under time shift
and we can thus define the alternative notation 
$T^{l_1,\ldots,l_{n-1}} = T_{j_1,\ldots,j_n}$.

Let us now mention two known and important information theoretic 
results for the Gaussian distribution. The first one is the formula of the Gaussian entropy:
\begin{itemize}
\item[]
{\it if
$p_{j_1,\ldots,j_n}$ is the $2 n$-dimensional Gaussian density with arbitrary mean and covariance
matrix $\Sigma_{j_1,\ldots,j_{n}}$, then the temporal entropy (\ref{e93}) is given by
}
\begin{align}									\label{e107}
T_{j_1,\ldots,j_{n}} & = \left\{ 1 + \log (2\pi) \right\} n + \frac{1}{2} \log \det
                        \u{\Sigma}_{j_1,\ldots,j_{n}}.
\end{align}
\end{itemize}
The second result is the maximum entropy property of the Gaussian distribution: 
\begin{itemize}
\item[]
{\it
among random vectors $\left(U_{j_1},\ldots,U_{j_n},V_{j_1},\ldots,V_{j_n}\right)$
having arbitrary density with fixed covariance matrix $\Sigma_{j_1,\ldots,j_n}$, 
the one that is normally distributed  
maximizes Shannon's entropy (\ref{e93}). The maximum of the entropy is given by
(\ref{e107}).
}
\end{itemize}

We now consider the previous constraints on the a.c.v.f. (\ref{e199}) and search for the
(strictly) stationary time series, with mean and pseudo-covariances null,  
that maximizes the temporal entropy.
\begin{theorem}[Maximal Shannon’s temporal entropy distribution]     		\label{t200}
\noindent
Consider the class of complex-valued and stationary time series
$\{ X_j \}_{j \in \Z}$ with mean null, variance $\sigma^2$, for some
$\sigma \in (0,\infty)$, and pseudo-covariances null.  
Denote by $\psi$ the a.c.v.f. of $\{ X_j \}_{j \in \Z}$, $\nu = \Re \psi$ and 
$\xi = \Im \psi$.
\begin{enumerate}
\item If the a.c.v.f. $\psi$ satisfies $\calC_r$ given in (\ref{e199}) 
or in (\ref{e201}), for $r=1,\ldots,k$,
thus 
$\psi(1) = \psi_1=\nu_1+ \i \,\xi_1,\ldots,\psi(k) = \psi_k=\nu_k + \i \, \xi_k$, 
then the time series $\{ X_j \}_{j \in \Z}$ in the above class that maximizes 
Shannon’s temporal entropy (\ref{e93}) 
with $n=k+1$ and $j_1=1,\ldots,j_{k+1}=k+1$
is the one for which the corresponding 
double f.d.d. (\ref{e42}) with $j_1=1,\ldots,j_{k+1}=k+1$ is Gaussian, 
with mean zero and with $2 (k+1) \times 2 (k+1)$ covariance matrix 
$\u{\Sigma}(k) = \u{\Sigma}^{1,\ldots,1}$ given by (\ref{e97}) with 
\begin{align*}
\psi_{UU} ( r ) = \psi_{VV} ( r )   = \frac{\nu_r}{2} \; 
\text{ and } \; \psi_{UV} ( r ) = \psi_{VU} ( -r )  = - \frac{\xi_r}{2},
\end{align*}
for $r=1,\ldots,k$. 
\item The corresponding value of the temporal entropy is given by
\begin{align*}
T(k) & = \left\{ 1 + \log (2\pi) \right\} (1 + k ) + \frac{1}{2} \log \det
                        \u{\Sigma}(k).
\end{align*}
\end{enumerate}
\end{theorem}
\begin{proof} 1.a.
This initial part of the proof shows that for any a.c.v.f. $\psi$, there exists a 
complex-valued Gaussian time series that is strictly
stationary, centered and radially symmetric. 
Let $n \ge 1$, $u_j, v_j \in \R$, $c_j = u_j - \i v_j$, for $j = 1, \ldots, n$,
let $j_1 < \ldots < j_n \in \Z$,
$\u{u} = (u_1,\ldots,u_n)^\top$, $\u{v} = (v_1,\ldots,v_n)^\top \in \R^n$ and define
\begin{align*}
q ( \u{u} , \u{v} ) & = \frac{1}{2} \sum_{l=1}^n \sum_{m=1}^n c_l \overline{c_m}
\psi ( j_l - j_m ).
\end{align*}
Then $q ( \u{u} , \u{v} ) \ge 0$ implies
\begin{align}                                                                           \label{e382}
q ( \u{u} , \u{v} ) & = \frac{1}{2} \sum_{l=1}^n \sum_{m=1}^n ( u_l - \i v_l )  ( u_m + \i v_m ) \{
                                \nu ( j_l - j_m ) + \i \xi ( j_l - j_m ) \} \nonumber \\
                    & = \frac{1}{2} \sum_{l=1}^n \sum_{m=1}^n ( u_l u_m + v_l v_m  ) \nu ( j_l - j_m ) -
                                          ( u_l v_m - v_l u_m  ) \xi ( j_l - j_m ).
\end{align}
Define
$\u{U} = (U_{j_1},\ldots,U_{j_n})^\top$ and 
$\u{V} = (V_{j_1},\ldots,V_{j_n})^\top$. Assume
$\left( \u{U}^\top, \u{V}^\top \right)$ normally distributed with mean zero and
covariance matrix $\u{\Sigma}_{j_1 , \ldots , j_n}$, as in (\ref{e97}).

A particular choice of $\u{\Sigma}_{j_1 , \ldots , j_n}$ can be obtained by setting
\begin{align*}                                                                  
\varphi ( \u{u}, \u{v} ) & =
\E \left[
\exp \left\{
\i \left( \u{u}^\top, \u{v}^\top \right)  { \u{U} \choose \u{V} }
\right\}
\right] = \exp \left\{ - \frac{1}{2} q ( \u{u} , \u{v} ) \right\},
\end{align*}
leading to
\begin{align*}
\left( \u{u}^\top, \u{v}^\top \right)
            \u{\Sigma}_{j_1 , \ldots , j_n} { \u{u} \choose \u{v} } 
& = q ( \u{u} , \u{v} ) .
\end{align*}
This, (\ref{e382}) and (\ref{e97}) yield
\begin{align}							\label{e122}                                                              
\E [ U_{j_l} U_{j_m} ] & = \frac{1}{2} \nu ( j_l - j_m ), \;
\E [ V_{j_l} V_{j_m} ] =   \frac{1}{2} \nu ( j_l - j_m ), \nonumber \\
\E [ U_{j_l} V_{j_m} ] & = - \frac{1}{2} \xi ( j_l - j_m ), \; 
\E [ V_{j_l} U_{j_m} ] = \frac{1}{2} \xi ( j_l - j_m )
\end{align}
and therefore $\xi ( j_l - j_m ) = - \xi ( j_m - j_l )$, 
for $l,m = 1, \ldots, n$.
Define
$\u{X} = (X_{j_1},\ldots,X_{j_n})^\top$, where
$X_{j_l} =  U_{j_l} + \i V_{j_l}$, for $l = 1, \ldots, n$.
We obtain the covariance matrix
\begin{align*}
\var ( \u{X} ) & = \E \left[ \u{X} \overline{\u{X}}^\top \right] 
        =  \E \left[ ( \u{U} + \i \u{V} ) ( \u{U} - \i \u{V} )^\top  \right] 
        =  \E \left[ \u{U} \u{U}^\top + \u{V} \u{V}^\top + \i \left(
                \u{V} \u{U}^\top - \u{U} \u{V}^\top   \right) \right]  \\
        & = \frac{1}{2} \left( \nu ( j_l - j_m ) + \nu ( j_l - j_m ) + \i \left[
                \xi ( j_l - j_m ) - \{ - \xi ( j_l - j_m ) \} \right] \right)_{l,m=1,\ldots,n} \\
        & = \left( \nu ( j_l - j_m ) + \i \xi ( j_l - j_m ) \right)_{l,m=1,\ldots,n} 
        = \left( \psi ( j_l - j_m ) \right)_{l,m=1,\ldots,n} 
\end{align*}
and the pseudo-covariance matrix
\begin{align*}
\E \left[ \u{X} \u{X}^\top \right] 
        & =  \E \left[ ( \u{U} + \i \u{V} ) ( \u{U} + \i \u{V} )^\top  \right] 
        =  \E \left[ \u{U} \u{U}^\top - \u{V} \u{V}^\top + \i \left(
                \u{V} \u{U}^\top + \u{U} \u{V}^\top   \right) \right]  \\
        & = \frac{1}{2} \left( \nu ( j_l - j_m ) - \nu ( j_l - j_m ) + \i \left[
                \xi ( j_l - j_m ) + \{ - \xi ( j_l - j_m ) \} \right] \right)_{l,m=1,\ldots,n} \\
        & = \left( 0 + \i 0 \right)_{l,m=1,\ldots,n} 
        = \u{0},
\end{align*} 
as desired. 
We have thus established the existence of a
complex-valued Gaussian time series $\{ X_t \}_{t \in \R}$ that is strictly
stationary, radially symmetric and centered. \\
1.b.
Consider $n=k+1$ and $j_1=1,\ldots,j_{k+1}=k+1$. Under
${\cal C}_r$, for $r=1,\ldots,k$, 
$\var ( \u{X} )$ is entirely determined: it is the $(k+1) \times (k+1)$, n.n.d. and Toeplitz matrix
(\ref{e206}).
The pseudo-covariance matrix and the mean vector are null and thus also determined.  
We know from (\ref{e122}) that, for $r=1,\ldots,k+1$, 
\begin{align*}
\psi_{UU}(r) = \psi_{UV}(r) = \frac{1}{2} \nu (r) = \frac{\nu_r}{2} 
\; \text{ and } \;
\psi_{UV}(r) = \psi_{VU}(-r) = -\frac{1}{2} \xi (r) = - \frac{\xi_r}{2}.
\end{align*} 
So the covariance matrix of $(\u{U}^\top,\u{V}^\top)$ is entirely determined 
by ${\cal C}_r$, for $r=1,\ldots,k$, and it is the $2 (k+1) \times 2 (k+1)$
matrix $\Sigma_{1,\ldots,k+1} = \Sigma^{1,\ldots,1}$. Clearly, 
$\E \left[ (\u{U}^\top,\u{V}^\top) \right] = 0$. 
The second information theoretic result 
for the Gaussian distribution, just above, concludes the 
proof Theorem \ref{t200}.1.
\\
2. The second information theoretic result for the Gaussian distribution, viz. 
(\ref{e107}), leads directly to the entropy formula in Theorem \ref{t200}.2.
\end{proof} 
So when $\{X_j\}_{j \in \Z}$ is the strictly stationary Gaussian-GvM time series, 
both spectral and temporal Shannon's entropies are maximized under the constraints
$\calC_r$, for $r=1,\ldots,k$.

\section{Some computational aspects}					\label{s3}

The following computational aspects are studied:
the computation of the integral functions of the $\gvm2$ time series in Section \ref{s31},
the estimation of the $\gvmk$ spectral distribution in Section \ref{s32} and the 
computation of the $\gvmk$ spectral d.f. in Section \ref{s33}.

\subsection{Integral functions of the $\mathbf{GvM_2}$ time series}	\label{s31} 

This section provides some series expansions for some integral functions
appearing with the $\gvmk$ spectral distribution.
Indeed, the information theoretic results of Section \ref{s2} require 
the constants or integral functions $G_r^{(k)}$, for $r = 0, \ldots , k$, and $H_r^{(k)}$, for
$r = 1, \ldots , k$. 
They are integrals over a bounded domain of smooth integrands 
and therefore numerical integration should perform well. Alternatively,
one can evaluate these integral functions by Fourier series expansions. 
Gatto (2007) provides expansions for some of these constants and in particular for 
those with $k=2$, that are given here.
Define 
$$ {\rm ep}_r = \left\{
\begin{array}{ll}
  1, &  {\rm if} \; r \; \text{ is even and positive}, \\
  0, &  {\rm otherwise}. 
\end{array}
\right.
$$
Let 
$\delta \in [0, \pi)$ and $\kappa_1,\kappa_2 \ge 0$. Then 
the following expansions hold for
$r=0,1,\ldots$,
\begin{align}                                                                        \label{e160}
G_r^{(2)}(\delta,\kappa_1,\kappa_2) = & \;   
I_0(\kappa_1) I_{\frac{r}{2}}(\kappa_2) \cos r \delta \; {\rm ep}_r 
+I_0(\kappa_2) I_r(\kappa_1) \nonumber 
\\ & + \sum_{j=1}^\infty \cos 2 j \delta \; I_j(\kappa_2) \left\{
I_{2 j + r}(\kappa_1) + I_{\mid 2 j - r \mid}(\kappa_1) \right\},
\end{align}
and
\begin{align}                                                                        \label{e166}
H_r^{(2)}(\delta,\kappa_1,\kappa_2) = &   
- I_0(\kappa_1) I_{\frac{r}{2}}(\kappa_2) \sin r \delta \; {\rm ep}_r
\nonumber \\ & 
+ \sum_{j=1}^\infty \sin 2 j \delta \; I_j(\kappa_2) \left\{
I_{2 j + r}(\kappa_1) - I_{\mid 2 j - r \mid}(\kappa_1) \right\}.
\end{align}
We can deduce from the two above expansions that
$G_r$ and $H_r$ inherit the asymptotic behavior of the Bessel function $I_r$, for large
$r$. It follows from Abramowitz and Stegun (1972, 9.6.10, p. 375)
that $I_r(z) = (z/2)^r \{ r \Gamma(r)\}^{-1}
\left\{1 + \O\left(r^{-1}\right)\right\}$, 
as $r \to \infty$. This and the Stirling approximation yield
$I_r(z) = ( 2 \pi r)^{-1/2}$ $ \{\e z / (2 r) \}^r
\left\{1 + \O\left(r^{-1}\right)\right\}$, as $r \to \infty$.
Hence $I_r$ decreases rapidly to zero as $r$ increases. 
We see that the same holds for $G_r$ and $H_r$.

\subsection{Estimation of the GvM spectral distribution}				\label{s32}

This section concerns the estimation problem. 
After reviewing some classical results of spectral estimation, it presents 
an estimator to the parameters of the GvM spectral distribution. 

A classical estimator of the spectral density is the periodogram. 
We remember that it is based on the discrete Fourier transform of the sample 
a.c.v.f., i.e. on    
$$
\Lambda_n (j)
= \sum_{r=-(n-1)}^{n-1} \hat{\psi}_n (r) \e^{- \i \frac{2 \pi j r}{n}},
$$
for $j = \lfloor (n-1)/2 \rfloor ,\ldots,-1, 1 ,\ldots , \lfloor
n/2 \rfloor$, $\hat{\psi}_n$ being the sample a.c.v.f. (\ref{e187}), $n$ the sample size and
$\lfloor \cdot \rfloor$ the floor function. Because of its nonparametric  
nature, the periodogram is well-suited for detecting particular features such as a periodicity, 
which may not be identified by a parametric estimator. 
However, its irregular nature may not be desirable in some contexts and it does not 
result from an important optimality criterion.

One of the earliest studies on maximum entropy spectral distributions
is Burg (1967), who considered $B(f_\sigma) = \int_{-\pi}^{\pi} \log f_\sigma(\theta) \d \theta$
as measure of entropy of the spectral density $f_\sigma$. This entropy is different than
our adaptation of Shannon's entropy, viz. $S(f_\sigma)$ given in (\ref{e92}). We can easily relate Burg's
entropy of the spectral density $f_\sigma$ to the Kullback-Leibler information as follows,
\begin{align}                                                                           \label{e85}
B ( f_\sigma ) & = 2 \pi \left\{ \log \frac{\sigma^2}{2 \pi} - \frac{1}{\sigma^2}
                        I ( u_\sigma | f_\sigma ) \right\}
                 = 2 \pi \left\{ \log \frac{\sigma^2}{2 \pi} -
                        I ( u_1 | f_1 ) \right\}.
\end{align}
This shows that
maximizing Burg's entropy amounts to minimize the re-directed Kullback-Leibler information,
instead of the usual Shannon's entropy.
For real-valued time series, it turns out that the spectral density estimator that maximizes the entropy
(\ref{e85}) subject to the constraints $\calC_r$, for $r = 1, \ldots, k$, with $k \le n - 1$, is equal
to the autoregressive estimator of order $k$. This autoregressive estimator 
is given by the formula of the spectral density the AR($k$) model which
has been fitted to the sample of $n$ consecutive 
values of the time series. For more details refer e.g. to
p. 365-366 of Brockwell and Davis (1991). 

Estimators of the parameters of the $\gvmk$ spectral distribution can be obtained from a
direct generalization of the trigonometric method of moments estimator for the $\gvmk$ 
circular distribution, which is introduced by Gatto (2008). This estimator is the
circular version of the method of moments estimators.
Consider the $\gvmk$ spectral distribution with unknown parameters $\mu_1,\ldots,\mu_k$ and
$\kappa_1,\ldots,\kappa_k$, for some $k \in \{ 1, \ldots, n-1 \}$. Consider the $r$-th a.c.v.f. condition $\calC_r$ in which
the spectral density $g_r$ is taken equal to the $\gvmk$ spectral density with variance $\sigma^2$, viz. 
$\sigma^2$ times the circular density (\ref{e94}), and in which the quantity $\psi_r \in \C$ is 
replaced by the sample a.c.v.f. at $r$, namely by $\hat{\psi}_n(r)$, $r = 1,\ldots,k$, cf. (\ref{e187}).
The resulting $r$-th equation can be re-expressed in a similar way to (\ref{e119}), which in turn leads to 
\begin{align}									\label{e163}
   \left( \begin{array}{c} \Re \hat{\psi}_n(r) \\ \Im \hat{\psi}_n(r)
   	  \end{array} \right) 
   = \sigma^2 \u{R}(r\mu_1)
   \left( \begin{array}{c} A_r^{(k)}(\delta_1,\ldots,\delta_{k-1},\kappa_1,\ldots,\kappa_{k}) \\
                           B_r^{(k)}(\delta_1,\ldots,\delta_{k-1},\kappa_1,\ldots,\kappa_{k})
   	  \end{array} \right),
\end{align}
for $r=1,\ldots,k$, with $k \le n-1$. This gives a system of $2 k$ real 
equations and $2 k$ unknown real parameter. The values 
of $\mu_1$, $\delta_1,\ldots,\delta_{k-1}$, $\kappa_1,\ldots,\kappa_{k}$
that solve this system of equations are the resulting estimators and they can be denoted
$\hat{\mu}_1$, $\hat{\delta}_1,\ldots,\hat{\delta}_{k-1}$, $\hat{\kappa}_1,\ldots,\hat{\kappa}_{k}$.
We now give two examples. 
\begin{example}[{\rm vM} spectrum]
When $k=1$ we have the basic vM unimodal spectral distribution.
Because $B_1^{(1)}(\kappa_1)=0$, we have
the estimating equation
\begin{align*}       
   \left( \begin{array}{c} \Re \hat{\psi}_n(1) \\ \Im \hat{\psi}_n (1)
          \end{array} \right)
   = \sigma^2 A_1^{(1)}(\kappa_1)
          \left( \begin{array}{c}
                        \cos \mu_1 \\ \sin \mu_1
          \end{array} \right)
   = \sigma^2 \frac{ I_1(\kappa_1) }{ I_0(\kappa_1) }
   \left( \begin{array}{c}
                        \cos \mu_1 \\ \sin \mu_1
          \end{array} \right),
\end{align*}
giving two equations and two unknown values, namely $\mu_1$ and $\kappa_1$.
The solutions are the estimators $\hat{\mu}_1$ and $\hat{\kappa}_1$.
For $\kappa_1>0$, if $\mu_1=0$ is given, then we have axial symmetry 
about the origin and so the corresponding time series is real-valued.
The two estimating equations reduce to the single equation
\begin{align}								\label{e185} 
\frac{\hat{\psi}_n(1)}{\sigma^2} & = A_1^{(1)}(\kappa_1),
\end{align}
whose solution is the estimator $\hat{\kappa}_1$. 
Note that Amos (1974) showed that $A_1^{(1)}$ has positive derivative over $(0,\infty)$.
It follows essentially from this fact that 
$A_1^{(1)}$ is a strictly increasing and differentiable 
probability d.f. over $[0,\infty)$. So its inverse function is easily computed. 
\end{example}
\begin{example}[$\gvm2$ spectrum]
When $k=2$ we retrieve the practical $\gvm2$ unimodal or bimodal spectral distribution.  
In the estimating equations (\ref{e163}) with $k=2$, we
can use the series expansions of the constants given by (\ref{e160}) and (\ref{e166}).
As previously mentioned, with $\kappa_1, \kappa_2 >0$, 
the $\gvm2$ distribution is axially symmetric 
around the axis $\mu_1$ iff
$\delta_1 = \delta^{(1)} = 0$ or $\delta_1 = \delta^{(2)} = \pi / 2$. 
With these values of $\delta_1$ and with
$\mu_1=0$, the axial symmetry is about the origin and so 
the corresponding time series is real-valued.
We note that 
$B_r^{(2)}\left(\delta^{(j)}, \kappa_1 , \kappa_2\right) = 0$, 
for $r = 1 , 2$ and for the cases $j=1,2$.
Because of these equalities and because $\Im \hat{\psi}_n(r) = 0$, for $r = 1 , 2$,
the estimating equations (\ref{e163}) simplify to
\begin{align*}
\frac{\hat{\psi}_n(r)}{\sigma^2} & =  A_r^{(2)} \left( \delta^{(j)},\kappa_1,\kappa_{2} \right), 
\end{align*} 
for $r = 1 , 2$ and for the two cases $j = 1 , 2$.
These estimating equations appear as the natural generalization of the
estimation equation (\ref{e185}) of the real-valued vM time series.
For each one of these two cases, we have two 
equations and two unknown values, namely $\kappa_1$ and $\kappa_2$. 
The solutions are the estimators $\hat{\kappa}_1$ and $\hat{\kappa}_2$.
\end{example}

\subsection{GvM spectral distribution function}					\label{s33} 

A formula for the $\gvmk$ spectral d.f. can be obtained in terms of a series as follows. 
Let $\psi^{(k)}$ denote the a.c.v.f. and let $f_\sigma^{(k)}$ denote the spectral density
of the $\gvmk$ time series with variance $\sigma^2$.
It follows from $\Re \psi^{(k)}(-r) = \Re \psi^{(k)}(r)$ and $\Im \psi^{(k)} (-r) =
- \Im \psi^{(k)}(r)$, for $r=1,2,\ldots$,
and from (\ref{e108}) that
\begin{eqnarray*}
\lefteqn{
f_\sigma^{(k)}(\theta | \mu_1,\ldots,\mu_k,\kappa_1,\ldots,\kappa_k) } \\
& = & \frac{1}{2 \pi} \sum_{r = -\infty}^{\infty} \psi^{(k)}(r) \exp \{ -\i  r \theta \} \\
& = & \frac{1}{2 \pi} \left( 1 + 2 \sum_{r=1}^{\infty}
\Re \psi^{(k)} (r)  \cos r \theta + \Im \psi^{(k)} (r) \sin r \theta \right) \\
& = &
\frac{\sigma^2}{2 \pi} \left( 1 + 2 \sum_{r=1}^{\infty} ( \cos r \theta , \sin r \theta) \u{R}(r\mu_1)
\left( \begin{array}{c} A_r^{(k)}(\delta_1,\ldots,\delta_{k-1},\kappa_1,\ldots,\kappa_{k}) \\
                        B_r^{(k)}(\delta_1,\ldots,\delta_{k-1},\kappa_1,\ldots,\kappa_{k})
\end{array} \right) \right),
\end{eqnarray*}
$\forall \theta \in (-\pi,\pi]$. Pointwise convergence is due to Dirichlet's theorem
(cf. e.g. Pinkus and Zafrany, 1997, p. 47).
The term by term integration of the Fourier series of a piecewise continuous function
converges uniformly towards the integral of the original function
(cf. e.g. Pinkus and Zafrany, 1997, p. 77).
So the $\gvmk$ spectral d.f. admits the series representation
\begin{eqnarray*}
\lefteqn{
F_\sigma^{(k)}(\theta|\mu_1,\ldots,\mu_k,\kappa_1,\ldots,\kappa_k) } \\
& = & \int_{-\pi}^\alpha f_\sigma(\alpha | \mu_1,\ldots,\mu_k,\kappa_1,\ldots,\kappa_k) \d \alpha \nonumber \\
& = & \frac{\sigma^2}{2 \pi} \Bigg( \theta + 2  \sum_{r=1}^{\infty}
	\frac{1}{r} \big[ A_r^{(k)}(\delta_1,\ldots,\delta_{k-1},\kappa_1,\ldots,\kappa_k) 
	\{ \sin r(\theta - \mu_1) + \sin r \mu_1 \} \nonumber \\
& & - B_r^{(k)}(\delta_1,\ldots,\delta_{k-1},\kappa_1,\ldots,\kappa_k) 
	\{ \cos r(\theta - \mu_1) - \cos r \mu_1 \} \big] \Bigg),
	\; \forall \theta \in (-\pi, \pi],
\end{eqnarray*}
where the convergence is uniform.
The order of the $r$-th summand is $r$
times smaller w.r.t. the original Fourier series, so we can expect rapid convergence.
When $k=2$, we can use the series expansions (\ref{e160}) and (\ref{e166}) for the computation this d.f. 
This expansion of the spectral d.f. can be used in conjunction with (\ref{e184}) for 
the computation of the $\L2$-norm of the increments of the spectral process
of the $\gvmk$ time series. 

\section{Concluding remarks}								\label{s4} 

This chapter presents an innovative application of the GvM distribution of directional
statistics to the analysis of stationary time series.
As already mentioned, these developments are merely a first analysis to 
the GvM and Gaussian-GvM time series.
The scope is limited to the presentation of some first few results on 
the GvM spectrum and further developments are required. For example,
simulation algorithms for the Gaussian-GvM time series could be developed. 
The GvM spectrum is motivated by rather theoretical considerations and one 
is aware that ad hoc spectra are often used in applied domains.
For example, the Pierson-Moskowitz spectrum is widely used for the 
stationary modelling of ocean waves, in the context of naval construction. We refer e.g. to 
p. 315-316 of Lindgren (2012) for a list of commonly used spectra (that includes the
Pierson-Moskowitz spectrum). These spectra correspond to continuous time stationary models, 
but wrapping them around the circle of circumference $2 \pi$ leads 
to the spectra of these processes sampled at integer times, viz. of time series.     
Lastly, we note that the connection between spectral and circular distributions is well-known:
the wrapped Cauchy circular distribution is the normalized spectral distribution of the  
the AR(1) time series and the cardioid distribution is the normalized spectral distribution of the 
first order moving average (MA(1)) time series. This connection is exploited by Tanigichi et al. (2020)
in order to construct new circular distributions. 
It is therefore in the opposite direction that this chapter exploits this connection.

\section{References}

\noindent
Abramowitz, M., Stegun, I. E. (1972), {\em Handbook of
Mathematical Functions with Formulas, Graphs, and Mathematical
Tables}, Dover Publications (9-th printing conform to the 10-th original printing).

\noindent
Amos, D. E. (1974), ``Computation of modified Bessel functions and their ratios'',
{\it Mathematics of Computation}, 28, 239-251.

\noindent 
Asmussen, S., Glynn P. W. (2007), {\it Stochastic Simulation. Algorithms and Analysis},
Springer.

\noindent
Astfalck, L., Cripps, E., Gosling, J., Hodkiewicz, M., Milne, I. (2018), 
``Expert elicitation of directional metocean parameters'', {\it Ocean Engineering}, 161, 
268-276.

\noindent
Bloomfield, P. (1973),
``An exponential model for the spectrum of a scalar time series'',
{\it Biometrika}, 60, 217-226.  

\noindent
Bogert, B. P., Healy, M. J. R., Tukey, J. W., Rosenblatt, M. (1963),
``The quefrency analysis of time series for echoes: cepstrum, pseudoauto-covariance, cross-cepstrum and saphe cracking'',
{\it Proceedings of the Symposium on Time Series Analysis},
editor Rosenblatt, Wiley and Sons, 209-243.  

\noindent 
Brillinger, D. R. (1993), ``The digital rainbow: some history and applications of numerical
spectrum analysis'', {\it Canadian Journal of Statistics}, 21, 1-19.

\noindent
Brockwell, P. J., Davis, R. A. (1991), {\it Time Series: Theory and Methods}, second edition, Springer.

\noindent 
Burg, J. P. (1967), ``Maximum entropy spectral analysis'', 
unpublished presentation, 37-th meeting of 
the Society of Exploration Geophysicists, Oklahoma City, Oklahoma.

\noindent
Burg, J. P. (1978), ``Maximum entropy spectral analysis'', {\it Modern Spectrum Analysis}, editor Childers.
Wiley and Sons, 34-41.

\noindent 
Chatfield, C. (2013), {\it The Analysis of Time Series. Theory and Practice}, Springer.

\noindent
Christmas, J. (2014), ``Bayesian Spectral Analysis With Student-$t$ Noise'',
{\it IEEE Transactions on Signal 
Processing}, 62, 2871-2878. 

\noindent
Cram\'er, H. (1942), ``On harmonic analysis in certain functional spaces'',
{\it Arkiv f\"or Matematik, Astronomi och Fysik}, 28B, 12.

\noindent
Cram\'er, H., Leadbetter, M. R. (1967), {\it Stationary and Related Stochastic Processes}, Wiley, New York.

\noindent 
Fisher, R. A. (1929), ``Tests of significance in harmonic analysis'',
{\it Proceedings of the Royal Society of London}, Series A, 125, 796, 54-59.

\noindent 
Gatto, R. (2008), ``Some computational aspects of the generalized von Mises distribution'', 
{\it Statistics and Computing}, 18, 321-331. 

\noindent
Gatto, R. (2009), ``Information theoretic results for circular distributions'', {\it Statistics}, 
43, 409-421.

\noindent
Gatto, R., Jammalamadaka, S. R. (2007),
``The generalized von Mises distribution'',
{\em Statistical Methodology}, 4, 341-353.

\noindent
Gatto, R., Jammalamadaka, S. R. (2015), ``Directional statistics: introduction'',
{\it StatsRef: Statistics Reference Online}, editors Balakrishnan et al., Wiley and Sons, 1-8.

\noindent
Gonella, J. (1972), ``A rotary-component method for analysing meteorological and oceanographic 
vector time series'', {\it Deep-Sea Research}, 19, 833-846.

\noindent
Huang, D. (1990), ``On the maximal entropy property for ARMA processes and ARMA approximation'',
{\it Advances in Applied Probability}, 612-626. 

\noindent 
Jammalamadaka, S. R., SenGupta, A.
(2001), {\em Topics in Circular Statistics}, World Scientific Press. 

\noindent
Kay, S. M., Marple, S. L. (1981), ``Spectrum analysis: a modern perspective'',
{\it Proceedings of the IEEE}, 69, 1380-1419. 

\noindent
Kim, S., SenGupta, A. (2013), ``A three-parameter generalized von Mises distribution'',
{\it Statistical Papers}, 54, 685-693. 

\noindent
Kullback, S. (1954), ``Certain inequalities in information theory and the Cram\'er-Rao
inequality'', {\em Annals of Mathematical Statistics}, 25, 745-751.

\noindent
Kullback, S., Leibler, R. A. (1951), ``On information and sufficiency'',
{\em Annals of Mathematical Statistics}, 22, 79-86.

\noindent
Laeri, F. (1990), ``Some comments on maximum entropy spectral analysis of time series'',
{\it Computers in Physics}, 4, 627-636.  

\noindent
Lin, Y., Dong, S. (2019), ``Wave energy assessment based on trivariate 
distribution of significant wave height, mean period and direction'', {\it 
Applied Ocean Research}, 87, 
47-63. 

\noindent
Lindgren, G. (2012),
{\it Stationary Stochastic Processes: Theory and Applications},
Chapman and Hall/CRC.

\noindent
Mardia, K. V., Jupp, P. E. (2000), {\em Directional Statistics}, Wiley.

\noindent 
Pinkus, A., Zafrany, S. (1997), {\em Fourier Series and Integral Transforms},
Cambridge University Press.

\noindent 
Rice, S. O. (1944), ``Mathematical analysis of random noise'', {\it Bell Systems technical Journal},
23, 282-332.
\\
Also in Rice (1954),
{\it Selected Papers on Noise and Stochastic Processes},
editor Wax, Dover Publications, p. 133-295.

\noindent
Rice, S. O. (1945), ``Mathematical analysis of random noise'', {\it Bell Systems technical Journal}, 24, 46-156.
\\
Also in Rice (1954), ``Mathematical analysis of random noise'',
{\it Selected Papers on Noise and Stochastic Processes},
editor Wax, Dover Publications, p. 133-295.

\noindent
Rowe, D. B. (2005), ``Modeling both the magnitude and phase of complex-valued fMRI data'',
{\it NeuroImage}, 25, 1310–1324.

\noindent Salvador, S., Gatto, R. (2020a),
``Bayesian tests of symmetry for the generalized von Mises distribution'',
preprint, Institute of mathematical statistics and actuarial science, University of Bern.

\noindent Salvador, S., Gatto, R. (2020b),
``A Bayesian test of bimodality for the generalized von Mises distribution'',
preprint, Institute of mathematical statistics and actuarial science, University of Bern.

\noindent
Shannon, C. E. (1948),
``A mathematical theory of communication'', {\em
Bell System Technical Journal}, 27,
379-423, 623-656.

\noindent
Taniguchi, M., Kato, S., Ogata, H., Pewsey, A. (2020),
``Models for circular data from time series spectra'', 
{\it Journal of Time Series Analysis}, 41, 808-829.   

\noindent
Yaglom, A. M. (1962),
{\it An Introduction to the Theory of Stationary Random Functions},
Prentice-Hall.

\noindent
Zhang, L., Li, Q., Guo, Y., Yang, Z., Zhang, L. (2018),  
``An investigation of wind direction and speed in a featured wind farm using joint 
probability distribution methods'', {\it Sustainability}, 10.

\end{document}